\documentclass{article}%
\usepackage{amsmath}
\usepackage{amsfonts}
\usepackage{amssymb}
\usepackage{graphicx}
\usepackage{geometry}
\usepackage{tikz}
\usetikzlibrary{decorations.markings}

\providecommand{\U}[1]{\protect\rule{.1in}{.1in}}

\usepackage{graphicx}
\newtheorem{theorem}{Theorem}

\newtheorem{corollary}[theorem]{Corollary}

\newtheorem{definition}[theorem]{Definition}

\newtheorem{lemma}[theorem]{Lemma}

\newtheorem{proposition}[theorem]{Proposition}

\newenvironment{proof}[1][Proof]{\noindent\textbf{#1.} }{\ \rule{0.5em}{0.5em}}
\begin{document}

\title{Planarity, duality and Laplacian congruence}
\author{Derek A. Smith and Lorenzo Traldi\\Lafayette College\\Easton Pennsylvania 18042
\and William Watkins\\California State University Northridge\\Northridge, California 91330}
\date{}
\maketitle


\begin{abstract}
We discuss the connections tying Laplacian matrices to abstract duality and
planarity of graphs.

AMS Subject Classification: 05C50

Keywords: congruence, cut, dual, flow, graph, Laplacian, lattice, matroid, planar, trace

\end{abstract}


\section{Introduction\label{sec:intro}}

The purpose of this paper is to discuss the relationships between Laplacian matrices and three important properties of graphs: planarity for individual graphs, 2-isomorphism (isomorphism of cycle matroids) for pairs of graphs, and duality for pairs of planar graphs. These ideas have long histories, and have been studied carefully and thoroughly for the better part of a century. The Laplacian matrix is motivated by Kirchoff's laws, which were formulated in the mid 1800s, before matrices were introduced. Planarity was studied by Kuratowski and Whitney in the 1920s and 1930s; Whitney also introduced the 2-isomorphism and duality relations at that time. The existence of a connection between Laplacian matrices and graph matroids is implicit in the matrix-tree theorem, which was formulated before World War II. (See \cite{KMPS} for some notes on the history of the matrix-tree theorem.) However it was not until the early 1990s that Watkins \cite{Wa1,Wa2} showed that the cycle matroid of a graph is determined by the Laplacian matrix, except for the fact that the Laplacian ignores loops.

\begin{theorem} (Watkins \cite{Wa1,Wa2})
\label{Watkins} 
If $G_1$ and $G_2$ are graphs with the same number of loops, then $G_1$ and $G_2$ are 2-isomorphic if and only if their reduced Laplacian matrices are congruent over $\mathbb Z$.
\end{theorem}

To keep the introduction readable, we delay technical definitions (2-iso\-morphism, congruence, reduced Laplacian, etc.)\ until later in the paper. 


In 1997, Bacher, de la Harpe and Nagnibeda \cite{BHN} proved that the cycle matroid of a graph is related to the lattices of cuts and flows. (See also the account of Godsil and Royle \cite[Chapter 14]{GR}.) Here are several of their results. The statements are augmented with obvious requirements involving bridges and loops.


\begin{theorem}(Bacher, de la Harpe and Nagnibeda
\cite{BHN})
\label{bhn}
Let $G_1$ and $G_2$ be graphs.
\begin{enumerate}
    \item If $G_1$ and $G_2$ are 2-isomorphic, then they have the same number of loops and their cut lattices are isomorphic.
    \item If $G_1$ and $G_2$ are 2-isomorphic, then they have the same number of bridges and their flow lattices are isomorphic.
    \item If $G_1$ and $G_2$ are dual planar graphs, then $G_1$ has the same number of loops as $G_2$ has bridges, and the cut lattice of $G_1$ is isomorphic to the flow lattice of $G_2$.
\end{enumerate}
\end{theorem}

Bacher, de la Harpe and Nagnibeda observed that the reduced Laplacians of a graph are Gram matrices for the cut lattice; see \cite[p.\ 183]{BHN}, although the term ``Laplacian'' does not appear there. This implies that two graphs have isomorphic cut lattices if and only if they have congruent reduced Laplacians, and this in turn implies that part 1 of Theorem \ref{bhn} is equivalent to one direction of Theorem \ref{Watkins}, and the converse of part 1 of Theorem \ref{bhn} is equivalent to the other direction of Theorem \ref{Watkins}.

Not realizing that Watkins had already proven a result equivalent to the converse of part 1, Bacher, de la Harpe and Nagnibeda left the converses of the three parts of Theorem \ref{bhn} as open problems in \cite{BHN}. In 2010, Su and Wagner \cite{SW} verified the converses of all three parts, and also extended the theory to the cut and flow lattices of regular matroids. 



The present paper grew out of our work on the problem of extending Theorem \ref{Watkins} to describe the relationship between Laplacian matrices of dual graphs. By the time we appreciated the connection with the work of Bacher, de la Harpe and Nagnibeda \cite{BHN} and Su and Wagner \cite{SW}, we had developed a set of ideas that includes our own proof of a result equivalent to part 3 of Theorem \ref{bhn}, and the converse. Our arguments are focused on matrices associated with graphs, rather than lattices associated with regular matroids. Of course a lattice may be identified with the collection of its Gram matrices, so the two approaches are not incompatible in theory; but the arguments are quite different.

The notion we develop is that in addition to having an unreduced Laplacian matrix that is uniquely defined up to simultaneous permutation of the rows and columns, a graph $G$ has ``unreduced dual Laplacian matrices,'' which are not uniquely defined but are congruent to each other over $\mathbb Z$. (The definition appears in Section \ref{sec:dlap}.) There are also reduced dual Laplacians, just as there are reduced Laplacians. Some properties of these matrices are analogues of properties of Laplacians; for instance the reduced dual Laplacian matrices of $G$ are Gram matrices for the flow lattice of $G$. 

Dual Laplacian matrices also have some properties that are rather different from properties of Laplacians. Two of these properties involve planarity and abstract duality.

\begin{theorem}
\label{trace}
Suppose $G$ is a graph with $m$ edges and $b$ bridges.
\begin{enumerate}
\item The trace of an unreduced dual Laplacian matrix of $G$ is an even integer, greater than or equal to $2(m-b)$.

\item $G$ is planar if and only if $G$ has an unreduced dual Laplacian matrix whose trace is equal to $2(m-b)$.
\end{enumerate}
\end{theorem}

Theorem \ref{trace} can be restated using the lattice terminology of Conway
\cite{C}: $G$ is planar if and only if its flow lattice has a superbase with a Gram matrix of the smallest possible trace.

\begin{theorem}
\label{duals}
Let $G$ be a planar graph. Then the following statements about a graph $G^{\ast}$ are equivalent.

\begin{enumerate}
\item $G$ and $G^{\ast}$ are abstract duals.

\item The number of loops in $G^{\ast}$ is the same as the number of bridges
in $G$, and a reduced Laplacian matrix of $G^{\ast}$ is a reduced dual Laplacian matrix of $G$.
\end{enumerate}
\end{theorem}

Here is an outline of the paper. In Section \ref{sec:lap} we summarize the connections tying congruence of Laplacian matrices to row equivalence of incidence matrices, and 2-isomorphism of graphs. A small example is presented in Section \ref{sec:exam}. Reduced and unreduced dual Laplacian matrices are defined in Section \ref{sec:dlap}, and the analogies between them and ordinary Laplacians are discussed. In Section \ref{sec:dlapgen} we
verify some properties of dual Laplacians. Theorems \ref{Watkins}, \ref{trace} and \ref{duals} are proven in Section
\ref{sec:center}; we also discuss a version of Theorem \ref{Watkins} for unreduced Laplacians, and we relate Theorem \ref{trace} to a famous planarity criterion of MacLane~\cite{Ma}. A couple of illustrative examples are presented in Section \ref{sec:examples}.

Before proceeding we should thank an anonymous reader for good advice, which significantly improved the readability of the paper.

\section{Some properties of Laplacian matrices}
\label{sec:lap}
 We use standard notation and terminology for graphs. A graph $G$ has a finite set $E(G)$ of edges, and a finite set $V(G)$ of vertices; we write $m=|E(G)|$ and $n=|V(G)|$. If $e \in E(G)$ is incident on $v,w \in V(G)$ then we write $e=vw$. If $e=vv$ then $e$ is a \emph{loop} at $v$; the number of loops in $G$ is denoted $\ell$. Two edges $e\neq e^{\prime}$ are \emph{parallel} if $e=vw$ and $e^{\prime}=vw$. A graph is \emph{simple} if it has neither loops nor parallels. 

\begin{definition}
\label{lapmatu}Let $G$ be a graph with $V(G)=\{v_{1},\dots,v_{n}\}$. Then the
\emph{Laplacian matrix} of $G$ is the $n\times n$ matrix with entries given
by
\[
L(G)_{ij}=%
\begin{cases}
- | \{ e \in E(G) \mid e=v_i v_j \}|\text{,} & \text{if }i\neq j\\
| \{ e \in E(G) \mid e=v_i v_k \text{ and } k \neq i \}|\text{,} & \text{if }i=j
\end{cases}
.
\]
\end{definition}

Six elementary properties of the Laplacian are immediately apparent from
Definition \ref{lapmatu}; we number them for ease of reference.

\textbf{Property I} $L(G)$ is a symmetric matrix with integer entries.

\textbf{Property II} $L(G)$ is not changed if loops are added to $G$ or
removed from $G$.

\textbf{Property III} If $G$ and $G^{\prime}$ are graphs then
$L(G)=L(G^{\prime})$ up to simultaneous permutation of the rows and columns
if, and only if, we obtain isomorphic graphs when we remove all loops from $G$
and $G^{\prime}$.

\textbf{Property IV} The sum of the columns of $L(G)$ is $\mathbf{0}$; and the same for
the rows.

Here the bold $\mathbf{0}$ denotes a matrix or vector whose entries all equal $0$.

\textbf{Property V} The trace $Tr(L(G))$ is $2(m-\ell)$.

\textbf{Property VI} If $G$ is a disconnected graph with connected components
$C_{1},\dots,C_{c(G)}$ then
\[
L(G)=%
\begin{pmatrix}
L(C_{1}) & \mathbf{0} & \mathbf{0}\\
\mathbf{0} & \ddots & \mathbf{0}\\
\mathbf{0} & \mathbf{0} & L(C_{c(G)})
\end{pmatrix}
\text{.}
\]

Properties IV and VI tell us that in each
connected component of $G$, the row of $L(G)$ corresponding to one vertex is the negative of the sum of the remaining rows. The same holds for the columns, of course.

\begin{definition}
\label{reducel}Let $V_{0}$ be a subset of $V(G)$, which contains precisely one
vertex from each connected component of $G$. The submatrix of $L(G)$ obtained
by removing all rows and columns corresponding to elements of $V_{0}$ is a
\emph{reduced} Laplacian of $G$, denoted $L_{V_{0}}(G)$.
\end{definition}

Reduced Laplacian matrices inherit properties I, II and VI directly from
$L(G)$. The reduced version of property V is an inequality: $Tr(L_{V_{0}%
}(G))<2(m-\ell)$ unless $\ell=m$, in which case $L_{V_{0}}(G)$ is the empty $0 \times 0$ matrix. The reduced version of property IV
is the famous matrix-tree theorem: $\det L_{V_{0}%
}(G)$ is the number of maximal forests of $G$. Details are given in many standard references, e.g.\ \cite[Theorem 13.2.1]{GR}.

The reduced version of property III is complicated by the arbitrary choice of
$V_{0}$. If $G$ and $G^{\prime}$ are graphs then these two statements are
equivalent: (a) when we adjoin a row and column to each of $L_{V_{0}%
}(G),L_{V_{0}^{\prime}}(G^{\prime})$ so that the row and column sums of both
matrices are $\mathbf{0}$, we obtain matrices that are equal up to simultaneous
permutation of the rows and columns; and (b) when we remove all loops from $G$
and $G^{\prime}$, identify all the vertices from $V_{0}$ to each other, and
identify all the vertices from $V_{0}^{\prime}$ to each other, we obtain
isomorphic connected graphs.

We also use fairly standard terminology when discussing matrices associated with graphs. Recall that a square matrix of integers $U$ is \emph{unimodular} or \emph{invertible over} $\mathbb Z$ if $\det U=\pm1$.

\begin{definition}
\label{sroweqdef}Two matrices $B$ and $B^{\prime}$ are \emph{strictly row
equivalent} over $\mathbb{Z}$ if and only if $B^{\prime}=UB$, where $U$ is unimodular.
\end{definition}

\begin{definition}
\label{roweqdef}Two matrices $B$ and $B^{\prime}$ are \emph{loosely row
equivalent} over $\mathbb{Z}$ if and only if
\[%
\begin{pmatrix}
B^{\prime}\\
\mathbf{0}
\end{pmatrix}
=U%
\begin{pmatrix}
B\\
\mathbf{0}
\end{pmatrix}
\text{,}%
\]
where $U$ is unimodular and the two $\bf 0$ submatrices may be of different sizes.
\end{definition}

Row equivalence can also be described using elementary operations. Two
matrices are strictly row equivalent over $\mathbb{Z}$ if and only if one can
be obtained from the other using some finite sequence of elementary row
operations over $\mathbb{Z}$, i.e., multiplying a row by $-1$, adding a
nonzero multiple of one row to another and permuting rows. For loose row
equivalence, it is also permissible to adjoin $\mathbf{0}$ rows, or remove them. A third way to describe row equivalence is that two $k$-column matrices are loosely row
equivalent if and only if their rows generate the same subgroup of
$\mathbb{Z}^{k}$. If two loosely row equivalent matrices have the same number
of rows, then the matrices are strictly row equivalent. (The last assertion
follows from properties of the Smith normal form of matrices with entries in
$\mathbb{Z}$, cf. \cite[Chapter 3]{J} for instance.)

\begin{definition}
\label{congdef}Two matrices $B$ and $B^{\prime}$ are \emph{congruent} over
$\mathbb{Z}$ if and only if $B^{\prime}=UBU^{T}$, where $U$ is unimodular and
$U^{T}$ is the transpose of $U$.
\end{definition}

\begin{definition}
\label{incidence}Let $\vec{G}$ denote an arbitrary directed version of $G$.
Then the \emph{incidence matrix} $N(\vec{G})$ is the $n\times m$ matrix whose
entries are given by the following.
\[
N(\vec{G})_{ve}=%
\begin{cases}
-1\text{,} & \text{if }v\text{ is the initial vertex of }e\text{, and }e \text{ is not a loop}\\
1\text{,} & \text{if }v\text{ is the terminal vertex of }e\text{, and }e \text{ is not a loop}\\
0\text{,} & \text{if }e\text{ is not incident on }v\text{, or }e\text{ is a
loop}%
\end{cases}
\]

\end{definition}

\begin{definition}
\label{reduce}Let $V_{0}$ be a subset of $V(G)$, which contains one vertex
from each connected component of $G$. Then the submatrix of $N(\vec{G})$
obtained by removing all rows corresponding to elements of $V_{0}$ is a
\emph{reduced} incidence matrix of $G$, denoted $N_{V_{0}}(\vec{G})$.
\end{definition}

The following equalities are immediate.

\textbf{Property VII} If $V_{0}\subseteq V(G)$ contains one vertex from each
connected component of $G$, then%
\[
N(\vec{G})\cdot N(\vec{G})^{T}=L(G)\text{ and }N_{V_{0}}(\vec{G})\cdot
N_{V_{0}}(\vec{G})^{T}=L_{V_{0}}(G)\text{.}%
\]

If $V_{0}^{\prime}$ is another such subset of $V(G)$ then $N_{V_{0}^{\prime}%
}(\vec{G})$ can be obtained from $N_{V_{0}}(\vec{G})$ as follows. For every
connected component of $G$ where $V_{0}$ contains a vertex $v$ and
$V_{0}^{\prime}$ contains a vertex $v^{\prime}\neq v$, (a) add all the other
rows of $N_{V_{0}}(\vec{G})$ corresponding to vertices from this connected
component to the $v^{\prime}$ row, (b) multiply the new row by $-1$, and (c)
label the new row with $v$ rather than $v^{\prime}$. If $U$ is the product of
elementary matrices corresponding to the row operations mentioned in (a) and
(b), then
\begin{equation}
U\cdot N_{V_{0}}(\vec{G})=N_{V_{0}^{\prime}}(\vec{G})\text{ and hence }U\cdot
L_{V_{0}}(G)\cdot U^{T}=L_{V_{0}^{\prime}}(G). \label{redcong}%
\end{equation}
We deduce the following elementary properties of the reduced matrices. Let
$V_{0}$ and $V_{0}^{\prime}$ be two subsets of $V(G)$, each of which contains precisely one vertex from each connected component of $G$.

\textbf{Property VIII} $N_{V_{0}}(\vec{G})$ and $N_{V_{0}^{\prime}}(\vec{G})$
are strictly row equivalent over $\mathbb{Z}$.

\textbf{Property IX} $L_{V_{0}}(G)$ and $L_{V_{0}^{\prime}}(G)$ are congruent
over $\mathbb{Z}.$

Notice that $\vec{G}$ appears in property VIII, while $G$ appears in property IX. The difference is that changing the direction of an edge $e$ does not affect $L_{V_{0}}(G)$, but it multiplies the $e$ column of $N_{V_{0}}(\vec{G})$ by $-1$.

Formula (\ref{redcong}) makes it clear that the row equivalence class of
$N_{V_{0}}(\vec{G})$ determines the congruence class of $L_{V_{0}}(G)$. A
natural question is this: does the congruence class of $L_{V_{0}}(G)$ also
determine the row equivalence class of $N_{V_{0}}(\vec{G})$? Property X tells
us that the answer is \textquotedblleft yes.\textquotedblright

\textbf{Property X} Let $G_{1}$ and $G_{2}$ be graphs with the same number of
loops, and for $i\in\{1,2\}$ let $V_{0i}\subseteq V(G_{i})$ be a subset that
contains one vertex from each connected component of each graph. Then any of
the following conditions implies the others.

\begin{enumerate}
\item $L_{V_{01}}(G_{1})$ and $L_{V_{02}}(G_{2})$ are congruent over
$\mathbb{Z}$.

\item There are oriented versions $\vec{G}_{1},\vec{G}_{2}$ and a bijection
$\beta:E(G_{1})\rightarrow E(G_{2})$ such that $N_{V_{01}}(\vec{G}_{1})$ and
$N_{V_{02}}(\vec{G}_{2})$ are strictly row equivalent over $\mathbb{Z}$, when
their columns are matched by $\beta.$

\item There are oriented versions $\vec{G}_{1},\vec{G}_{2}$ and a bijection
$\beta:E(G_{1})\rightarrow E(G_{2})$ such that $N(\vec{G}_{1})$ and $N(\vec
{G}_{2})$ are loosely row equivalent over $\mathbb{Z}$, when their columns are
matched by $\beta.$

\item $G_{1}$ and $G_{2}$ are 2-isomorphic. (I.e., their cycle matroids are isomorphic.)
\end{enumerate}

The equivalence among conditions 2, 3
and 4 of property X is a famous theorem of Whitney \cite{W1},
and there are many expositions in the literature. For instance, a thorough
discussion is provided by Oxley \cite[Chapter 5]{O}. Note that the phrase ``over $\mathbb{Z}$'' is not important in conditions 2--4; these conditions remain equivalent if
$\mathbb{Z}$ is replaced by a field. In fact most textbook presentations of
the theory of incidence matrices, like those in \cite[Chapter 2]{B},
\cite[Chapter 8]{GR} and \cite[Chapter 5]{O}, are formally restricted to
fields; however the presentations are easily modified to work over
$\mathbb{Z}$.

The fact that condition 1 of property X is equivalent to the other conditions is due to Watkins \cite{Wa1, Wa2}; this is Theorem \ref{Watkins} of the introduction. It is important to realize that ``over $\mathbb{Z}$'' is crucial in condition 1. In fact, condition 1 is not equivalent to the other conditions for any nontrivial graph over any field. For if $F$ is a field, $G$ is a nontrivial
graph, $a>1$ is an integer not divisible by the characteristic of $F$ and $a^{2}G$ is the graph obtained by replacing each edge of $G$ with $a^{2}$
parallel edges, then $G$ and $a^{2}G$ are certainly not 2-isomorphic. However%
\[
L_{V_{0}}(a^{2}G)=a^{2}L_{V_{0}}(G)=(aI)L_{V_{0}}(G)(aI)=(aI)L_{V_{0}}(G)(aI)^{T}
\]
and as $aI$ is invertible over $F$, it follows that $L_{V_{0}}(a^{2}G)$ is congruent to $L_{V_{0}}(G)$ over $F$. More details
of property X, including a proof of the equivalence of condition 1 with
conditions 2--4, are discussed in Section \ref{sec:center}.

There are several equivalent ways to describe 2-isomorphisms. Two of them are stated in Definition~\ref{twoiso}. We refer to Oxley \cite{O} for other equivalent descriptions, and a thorough account of their properties.

\begin{definition}
\label{twoiso}Let $G_{1}$ and $G_{2}$ be graphs. Then a bijection $\beta:E(G_{1})\rightarrow E(G_{2})$ is a \emph{2-isomorphism} if it defines an isomorphism between the cycle matroids of $G_1$ and $G_2$. That is, $\beta$ satisfies the following equivalent conditions.
\begin{enumerate}
    \item A subset $S \subseteq E(G_1)$ is the edge set of a maximal forest of $G_1$ if and only if $\beta(S)$ is the edge set of a maximal forest of $G_2$.
    \item There are oriented
versions $\vec{G}_{1}$ and$\ \vec{G}_{2}$ such that vectors in $\mathbb{Z}%
^{E(G_{1})}$ corresponding to circuits of $G_{1}$ are matched by $\beta$ to
vectors in $\mathbb{Z}^{E(G_{2})}$ corresponding to circuits of $G_{2}$.
\end{enumerate}
\end{definition}

Recall that a circuit in a graph is a minimal closed path. The vector
corresponding to a circuit is obtained by following the circuit according to
one of the two orientations, and placing $\pm1$ in the $e$ coordinate of the
vector for each edge $e$ that appears on the circuit, with $+1$ (resp. $-1$)
representing agreement (resp. disagreement) between the $\vec{G}$ direction of
$e$ and the direction of $e$ on the circuit. The subgroup of $\mathbb{Z}^{E(G)}$ generated by these vectors is called the \emph{cycle group} of $G$, or the \emph{lattice of integral flows} of $G$.

The last property we discuss concerns the relationship between Laplacian matrices and maximal forests.

If $M$ is a maximal forest of $G$ and $\vec{M}$ inherits edge directions from
$\vec{G}$ then the matrix-tree theorem tells us that $N_{V_{0}}(\vec{M})$ is a
unimodular submatrix of $N_{V_{0}}(\vec{G})$, which includes the columns
corresponding to edges of $M$. For convenience we adopt a notational
shorthand: if $\vec{G}-E(M)$ is the directed graph obtained from $\vec{G}$ by
removing all the edges of $M$, then we define
\[
C(M):=N_{V_{0}}(\vec{M})^{-1}\cdot N_{V_{0}}(\vec{G}-E(M)).
\]
This useful matrix appears in several references \cite{BHN, GR, O, SW}, but it
does not seem to have a standard name. We use the letter $C$ because the rows
represent the fundamental cuts of $G$ with respect to $M$. More information
about $C(M)$ is given in Sections \ref{sec:dlap} and \ref{sec:dlapgen}.

If $\ell=m$ or $G$ is a forest then $C(M)$ is the empty $0\times0$ matrix; otherwise,
$C(M)$ is an $(n-c(G))\times(m-n+c(G))$ matrix. Of course the $C(M)$ notation
is incomplete, as it does not mention $\vec{G}$ or $V_{0}$. Notice also that%
\[
N_{V_{0}}(\vec{G})=N_{V_{0}}(\vec{M})\cdot%
\begin{pmatrix}
I & C(M)
\end{pmatrix}
\cdot P_{M}\text{,}%
\]
where $I$ is an identity matrix of order $n-c(G)$ and $P_{M}$ is a permutation
matrix that permutes the columns of $%
\begin{pmatrix}
I & C(M)
\end{pmatrix}
$ into the order of $E(G)$ used for the columns of $N_{V_{0}}(\vec{G})$.
Permutation matrices satisfy $P_{M}P_{M}^{T}=I$, so
\[
L_{V_{0}}(G)=N_{V_{0}}(\vec{G})\cdot N_{V_{0}}(\vec{G})^{T}=N_{V_{0}}(\vec
{M})\cdot\left(  I+C(M)C(M)^{T}\right)  \cdot N_{V_{0}}(\vec{M})^{T}\text{.}%
\]
If $G=M$ is a forest then $C(M)$ is empty, and the equation holds with $I+C(M) C(M)^T = I$. If $m=\ell$, on the other hand, then the equation holds vacuously -- all the matrices are empty.

We deduce the following.

\textbf{Property XI} The congruence class of $I+C(M)C(M)^{T}$ over
$\mathbb{Z}$ is the same as that of $L_{V_{0}}(G)$.

Property XI\ tells us that we may think of the reduced forms of properties I -- X as applying to $I+C(M)C(M)^{T}$ matrices rather than $L_{V_{0}}(G)$
matrices, up to congruence over $\mathbb{Z}$. For instance the equivalence
between conditions 4 and 1 of property X may be rephrased like this: $G_{1}$
and $G_{2}$ are 2-isomorphic if and only if there are oriented versions
$\vec{G}_{i}$, maximal forests $M_{i}$, and subsets $V_{0i}\subseteq V(G_{i})$
containing one vertex from each connected component of each graph, such that
$I+C(M_{1})C(M_{1})^{T}$ and $I+C(M_{2})C(M_{2})^{T}$ are congruent over
$\mathbb{Z}$.

\section{An example}
\label{sec:exam}

Before discussing dual Laplacian matrices, we consider a small example. 

Suppose $G$ has two vertices and three parallel non-loop
edges. A maximal forest $M$ of $G$ consists of one edge. Depending on the edge directions, $C(M)$ is $%
\begin{pmatrix}
1 & 1
\end{pmatrix}
$, $%
\begin{pmatrix}
1 & -1
\end{pmatrix}
$, $%
\begin{pmatrix}
-1 & 1
\end{pmatrix}
$ or $%
\begin{pmatrix}
-1 & -1
\end{pmatrix}
$. In every case, $I + C(M) C(M)^T$ is the $1 \times 1$ matrix whose only entry is $3$; this is the same as $L_{V_0}(G)$.

Foreshadowing the results of the next sections, notice that if $I'$ is the $2 \times 2$ identity matrix then $I'+C(M)^T C(M)$ is one of these two matrices.
\[%
\begin{pmatrix}
2 & -1 \\
-1 & 2
\end{pmatrix}
\qquad \qquad
\begin{pmatrix}
2 & 1 \\
1 & 2 
\end{pmatrix}
\]
The first matrix is the reduced Laplacian of $K_{3}$, the dual graph of $G$. The trace of the unreduced Laplacian $L(K_3)$ is $6$. In contrast, the second matrix is not a reduced Laplacian of any graph. Moreover, if we enlarge this matrix to a $3 \times 3$ matrix whose rows and columns sum to $\bf 0$ we get a matrix of trace $10$,
\[
\begin{pmatrix}
2 & 1 & -3 \\
1 & 2 & -3 \\
-3 & -3 & 6
\end{pmatrix}
\text{.}
\]

\section{Dual Laplacian matrices}
\label{sec:dlap}

Here is another famous definition of Whitney \cite{W2}; again, we refer to
Oxley \cite{O} for a thorough discussion.

\begin{definition}
\label{abduals}
Let $G_{1}$ and $G_{2}$ be graphs. Then $G_{1}$ and $G_{2}$ are \emph{abstract duals} if and only if there is a bijection $\beta:E(G_{1})\rightarrow E(G_{2})$ that defines an isomorphism between the cycle matroid of $G_1$ and the bond matroid of $G_2$. That is, $\beta$ satisfies the following equivalent conditions.
\begin{enumerate}
    \item A subset $S \subseteq E(G_1)$ is the edge set of a maximal forest of $G_1$ if and only if $\beta(S)$ is the complement of the edge set of a maximal forest of $G_2$.
    \item There are
oriented versions $\vec{G}_{1}$ and$\ \vec{G}_{2}$ such that vectors in
$\mathbb{Z}^{E(G_{1})}$ corresponding to circuits of $G_{1}$ are matched by
$\beta$ to vectors in $\mathbb{Z}^{E(G_{2})}$ corresponding to edge cuts of
$G_{2}$.
\end{enumerate}
\end{definition}

If $W$ is a proper subset of $V(G)$ then the vector corresponding to the
edge cut determined by $W$ is obtained by placing $\pm1$ in the $e$ coordinate
of the vector for each non-loop edge $e$ that is incident on just one vertex
of $W$, with $+1$ (resp. $-1$) representing an edge directed toward $W$ (resp.
away from $W$) in $\vec{G}$. The subgroup of $\mathbb{Z}^{E(G)}$ generated by these vectors is called the \emph{cut group} of $G$, or the \emph{lattice of integral cuts} of $G$.

It is easy to see from Definitions \ref{twoiso} and \ref{abduals} that there
is a strong connection between 2-isomorphism and abstract duality: if $G_{1}$
and $G_{2}$ are abstract duals, then every graph 2-isomorphic to $G_{1}$ is an
abstract dual of every graph 2-isomorphic to $G_{2}$. It is not so easy to see
another famous theorem of Whitney \cite{W2}: $G$ has an abstract dual if and
only if $G$ is planar.

It turns out that if $I^{\prime}$ is an identity matrix of order $m-n+c(G)$,
then almost all of the fundamental properties of Laplacians listed in Section \ref{sec:lap} have analogues for matrices of the form $I^{\prime}+C(M)^{T}C(M)$. 


\begin{definition}
\label{lapdual}If $G$ is a graph with a maximal forest $M$ then any matrix
congruent over $\mathbb{Z}$ to $I^{\prime}+C(M)^{T}C(M)$ is a \emph{reduced
dual Laplacian }matrix of $G$.
\end{definition}

When $C(M)$ is the empty matrix -- i.e., when $m = \ell$ or $G$ is a forest -- the matrix $I^{\prime}+C(M)^{T}C(M)$ of Definition \ref{lapmatu} should be interpreted as being equal to $I'$. It is empty if $G$ is a forest.

\begin{definition}
\label{unredlapdual}If $G$ is a graph then an \emph{unreduced} dual Laplacian
matrix of $G$ is obtained by adjoining a row and column to a reduced dual
Laplacian matrix of $G$, so that the rows and columns of the
resulting matrix sum to $\mathbf{0}$.
\end{definition}

Notice that compared to Definitions \ref{lapmatu} and \ref{reducel},
Definitions \ref{lapdual} and \ref{unredlapdual} are \textquotedblleft
backward\textquotedblright: we start with reduced dual Laplacian matrices,
and construct unreduced dual Laplacian matrices by enlarging the reduced ones.
To make sure there is no misunderstanding we should emphasize that dual
Laplacians do not require dual graphs: every graph has reduced and unreduced
dual Laplacian matrices, whether the graph is planar or nonplanar. If $G=M$ is
a forest, the only reduced dual Laplacian matrix of $G$ is the empty
$0\times0$ matrix, and the only unreduced dual Laplacian matrix of $G$ is the
$1\times1$ matrix $\mathbf{0}$. Otherwise the reduced dual Laplacian matrices of $G$
are symmetric $(m-n+c(G))\times(m-n+c(G))$ matrices.

In general we use $^{\ast}$ to indicate dual Laplacian matrices and their
properties. For instance $L^{\ast}_{V_{0}}(G)$ denotes a reduced dual
Laplacian of $G$ obtained using $V_{0}$, and $L^{\ast}(G)$ denotes an
unreduced dual Laplacian matrix of $G$. It is important to keep in mind that
unlike $L_{V_{0}}(G)$ and $L(G)$, the notations $L_{V_{0}}^{\ast}(G),L^{\ast
}(G)$ are not well defined. In consequence there is no
property III$^{\ast}$. However, we will see in Section \ref{sec:dlapgen} that
these matrices satisfy the following property.

\textbf{Property IX}$^{\ast}$ $L_{V_{0}}^{\ast}(G)$ and $L^{\ast}(G)$ are well
defined up to congruence over $\mathbb{Z}$.{}

That is, the reduced dual Laplacian matrices of $G$ are all congruent over $\mathbb{Z}$, and the unreduced dual Laplacian matrices of $G$ are all congruent over $\mathbb{Z}$.

Here are some other properties of dual Laplacian matrices.

\textbf{Property I}$^{\ast}$ $L^{\ast}(G)$ and $L_{V_{0}}^{\ast}(G)$\ are
symmetric matrices with integer entries.

\textbf{Property II}$^{\ast}$ $L^{\ast}(G)$ and $L_{V_{0}}^{\ast}(G)$ are not
changed if bridges are added to $G$ or removed from $G$.

\textbf{Property IV}$^{\ast}$ The sum of the columns of $L^{\ast}(G)$ is $\mathbf{0}$;
and the same for the rows. The reduced version of property IV$^{\ast}$ is that
reduced dual Laplacian matrices satisfy the matrix-tree theorem, just as
reduced Laplacian matrices do. That is, $\det(I^{\prime}+C(M)^{T}C(M)) $ is the number of maximal forests of $G$ \cite[Theorem 14.7.3]{GR}.

Property II$^{\ast}$ implies that the dual version of property VI is rather
different from the original:

\textbf{Property VI}$^{\ast}$ If $G$ is not connected then any connected graph
obtained by adding bridges to $G$ has the same $L^{\ast}$ and $L_{V_{0}%
}^{\ast}$ matrices as $G$.

Before stating a property VII$^{\ast}$, it is helpful to discuss $C(M)$ a
little more. Recall that $C(M)=N_{V_{0}}(\vec{M})^{-1}\cdot N_{V_{0}}(\vec
{G}-E(M))$. The rows of $C(M)$ correspond to the rows of $N_{V_{0}}(\vec
{M})^{-1}$, which are indexed by the same set that indexes the columns of
$N_{V_{0}}(\vec{M})$, i.e., $E(M)$. The columns of $C(M)$ correspond to the
columns of $N_{V_{0}}(\vec{G}-E(M))$, which are indexed by the edges of
$G-E(M)$.

Now, consider the matrix $%
\begin{pmatrix}
C(M)^{T} & -I^{\prime}%
\end{pmatrix}
$. The columns of $I^{\prime}$ inherit an indexing from the rows of $C(M)^{T}%
$, which are the columns of $C(M)$; so the columns of $I^{\prime}$ are indexed
by the edges of $G-E(M)$. Of course the columns of $C(M)^{T}$ are the rows of
$C(M)$, and as was just discussed they are indexed by $E(M)$. All in all,
then, the columns of $%
\begin{pmatrix}
C(M)^{T} & -I^{\prime}%
\end{pmatrix}
$ are indexed by the edges of $G$. We define $F(M)$ to be the matrix
\[
F(M):=%
\begin{pmatrix}
C(M)^{T} & -I^{\prime}%
\end{pmatrix}
\cdot P_{M}\text{,}%
\]
where $P_{M}$ is a permutation matrix that permutes the columns of 
\[
\begin{pmatrix}
C(M)^{T} & -I^{\prime}
\end{pmatrix}
\]
into the order of $E(G)$ used for the columns of $N_{V_{0}}(\vec{G})$, as
before. N.b.\ Like $C(M)$, the notation $F(M)$ is incomplete; it does not
mention either $\vec{G}$ or $V_{0}$.

It turns out that $F(M)$ plays a role dual to that of $N_{V_{0}}(\vec{G})$. We
give more details in Section \ref{sec:dlapgen}, but we can certainly see the following.

\textbf{Property VII}$^{\ast}$ If $\widehat{F}(M)$ is the matrix obtained from
$F(M)$ by adjoining a new row equal to the negative of the sum of the rows of
$F(M)$, then%
\[
F(M)F(M)^{T}=I^{\prime}+C(M)^{T}C(M)=L_{V_{0}}^{\ast}(G)\text{ \ and
\ }\widehat{F}(M)\widehat{F}(M)^{T}=L^{\ast}(G)\text{.}%
\]

Recall that properties VIII and IX differ in that the former involves $\vec{G}$ and the latter involves $G$. In Section \ref{sec:dlapgen} we see that there is an analogous difference between properties VIII$^{\ast}$ and IX$^{\ast}$.

\textbf{Property VIII}$^{\ast}$ For a fixed choice of edge directions in
$\vec{G}$, the $F(M)$ matrices that arise from different choices of $M$ and
$V_{0}$ are all strictly row equivalent to each other over $\mathbb{Z}$.

Before proceeding we take a moment to describe the effect on $F(M)$ of changing the direction of an edge $e$, while holding $M$ and $V_0$ fixed. (a) If $e \notin E(M)$, then reversing the direction of $e$ has the effect of multiplying the $e$ column of $N_{V_{0}}(\vec
{G}-E(M))$ by $-1$. This in turn has the effect of multiplying the $e$ column of $C(M)=N_{V_{0}}(\vec{M})^{-1}\cdot N_{V_{0}}(\vec
{G}-E(M))$ by $-1$. The effect on
$
F(M)=%
\begin{pmatrix}
C(M)^{T} & -I^{\prime}%
\end{pmatrix}
\cdot P_{M}%
$
is to multiply the $e$ row of $C(M)^T$ by $-1$ while leaving the $-I'$ block of $F(M)$ unchanged. Of course multiplying \emph{part} of a row by $-1$ is not an elementary row operation. (b) If $e \in E(M)$, then reversing the direction of $e$ has the effect of multiplying the $e$ column of $N_{V_{0}}(\vec{M})$ by $-1$. This in turn has the effect of multiplying the $e$ row of $C(M)=N_{V_{0}}(\vec{M})^{-1}\cdot N_{V_{0}}(\vec
{G}-E(M))$ by $-1$. The effect on $F(M)$ is to multiply the $e$ column of $C(M)^T$ by $-1$. Again, this effect is not an elementary row operation. These observations explain why property VIII$^{\ast}$ requires a fixed choice of edge directions.

On the other hand, Property IX$^{\ast}$ does not require a fixed choice of edge directions. The preceding paragraph gives two reasons for this. (a) If $e \notin E(M)$ then multiplying the $e$ column of $C(M)$ by $-1$ has the effect of replacing $C(M)$ with $C(M)U$, where $U$ is the elementary matrix corresponding to the column multiplication. As $U^TU=I'$, the effect on $I^{\prime}+C(M)^{T}C(M)$ is to replace it with
\[
I^{\prime}+U^TC(M)^{T}C(M)U=U^T \cdot (I^{\prime}+C(M)^{T}C(M)) \cdot U\text{,}
\]
which is congruent to $I^{\prime}+C(M)^{T}C(M)$ over $\mathbb{Z}$. (b) If $e \in E(M)$ then multiplying the $e$ row of $C(M)$ by $-1$ has no effect on $I^{\prime}+C(M)^{T}C(M)$.

So far, we have stated properties I$^{\ast}$, II$^{\ast}$, IV$^{\ast}$,
VI$^{\ast}$, VII$^{\ast}$, VIII$^{\ast}$ and IX$^{\ast}$. There is no property
III$^{\ast}$, and property XI$^{\ast}$ is Definition \ref{lapdual}. The two remaining dual properties, V$^{\ast}$ and X$^{\ast}$, are Theorems \ref{trace} and \ref{duals}, stated in the introduction.

\section{Properties II$^{\ast}$, VIII$^{\ast}$ and IX$^{\ast}$%
\label{sec:dlapgen}}

Let $\vec{G}$ be a directed version of a graph $G$, and $M$ a maximal forest
of $G$. Let $I^{\prime}$ be the identity matrix of order $m-n+c(G)$, and let
$F(M)$ be the matrix
\[
F(M)=%
\begin{pmatrix}
C(M)^{T} & -I^{\prime}%
\end{pmatrix}
\cdot P_{M}%
\]
mentioned above. Then $F(M)F(M)^{T}=I^{\prime}+C(M)^{T}C(M)$ is a reduced dual
Laplacian matrix of $G$.\ Also
\begin{align*}
F(M)\cdot N_{V_{0}}(\vec{G})^{T}  &  =%
\begin{pmatrix}
C(M)^{T} & -I^{\prime}%
\end{pmatrix}
\cdot P_{M}P_{M}^{T}\cdot%
\begin{pmatrix}
I\\
C(M)^{T}%
\end{pmatrix}
\cdot N_{V_{0}}(\vec{M})^{T}\\
&  =\left(  C(M)^{T}-C(M)^{T}\right)  \cdot N_{V_{0}}(\vec{M})^{T}=\mathbf{0}\text{,}%
\end{align*}
so each row of $F(M)$ is orthogonal to all the rows of $N_{V_{0}}(\vec{G})$.

Notice that $F(M)$ is an $(m-n+c(G)) \times m$ matrix and $N_{V_{0}}(\vec{G})$ is an $(n-c(G)) \times m$ matrix. Both matrices have linearly independent rows, so it follows that the row spaces of these two matrices are orthogonal complements in the vector space $\mathbb{Q}^m$. Because of the $I$ and $-I'$ blocks of
\[
N_{V_{0}}(\vec{M})^{-1} \cdot N_{V_{0}}(\vec{G})\cdot P_{M}^{-1}=%
\begin{pmatrix}
I & C(M)
\end{pmatrix}
\text{  and  }%
F(M)\cdot P_{M}^{-1}=%
\begin{pmatrix}
C(M)^{T} & -I^{\prime}%
\end{pmatrix}
\text{,}%
\]
it is easy to deduce that the groups generated by the rows of $F(M)$ and $N_{V_{0}}(\vec{G})$ are orthogonal complements in the free
abelian group $\mathbb{Z}^{E(G)}$. That is, the rows of $F(M)$ generate the cycle group (also called the lattice of integral flows) of $G$. The fact
that the rows of $F(M)$ represent a basis of the cycle group implies directly
that $F(M)F(M)^{T}$ is a Gram matrix for the lattice of integral flows of $G$,
as mentioned by Godsil and Royle \cite[Chapter 14]{GR}.

Each row of $F(M)$ has precisely one nonzero entry from $I^{\prime}$, so each row of $F(M)$ corresponds to a circuit of $G$ that includes precisely one edge outside $M$. That is, the rows of $F(M)$ represent the \emph{fundamental circuits} of $G$ with respect to $M$. The observation of the preceding paragraph -- that the cycle group of $G$ is generated by the fundamental circuits with respect to $M$, for every maximal forest $M$ -- is a well-known elementary property of the fundamental circuits. In textbooks of graph theory or matroid theory like
\cite{B} or \cite{O}, this elementary property of fundamental circuits is often
stated only for cycle spaces defined over fields; but as noted above it is easy to deduce the integral version, because of the $I$ and $-I'$ blocks in the matrices. The statement over
$\mathbb{Z}$ is more common in textbooks of algebraic topology, like \cite{M}; it is also discussed by Bacher, de la Harpe and Nagnibeda \cite[Lemma 2]{BHN}.

The same cycle group is generated by the rows of $F(M)$, independent of the choices of $M$ and $V_0$. We deduce property VIII$^{\ast}$: All of the $F(M)$ matrices associated with $\vec{G}$ are strictly row equivalent over $\mathbb{Z}$.

That is, if $M$ and $M^{\prime}$ are maximal forests of $G$ then
$UF(M)=F(M^{\prime})$ for some unimodular matrix $U$. It follows that%
\begin{align}
U\left(  I^{\prime}+C(M)^{T}C(M)\right)  U^{T}  &  =UF(M)F(M)^{T}%
U^{T}\label{dcong}\\
&  =F(M^{\prime})F(M^{\prime})^{T}=I^{\prime}+C(M^{\prime})^{T}C(M^{\prime
})\text{,}\nonumber
\end{align}
so $I^{\prime}+C(M)^{T}C(M)$ and $I^{\prime}+C(M^{\prime})^{T}C(M^{\prime})$
are congruent over $\mathbb{Z}$. We conclude that all the reduced dual Laplacian matrices of $\vec{G}$ provided by Definition \ref{lapdual} are congruent
to each other over $\mathbb{Z}$. As discussed at the end of Section \ref{sec:dlap}, changing edge directions does not affect reduced dual Laplacian matrices, up to congruence; it follows that $\vec G$ may be replaced by $G$ in the preceding sentence. This is the reduced form of property IX$^{\ast}$.

For the unreduced form of property IX$^{\ast}$, notice that (\ref{dcong})
implies that if $L^{\ast}(G)$ and $L^{\prime\ast}(G)$ are the matrices
obtained from $I^{\prime}+C(M)^{T}C(M)$ and $I^{\prime}+C(M^{\prime}%
)^{T}C(M^{\prime})$ (respectively) by adjoining a new row and column so that
the row and column sums are $\mathbf{0}$, then%
\[
L^{\prime\ast}(G) =%
\begin{pmatrix}
1 & -\mathbf{1}\\
\mathbf{0} & I^{\prime}%
\end{pmatrix}%
\begin{pmatrix}
0 & \mathbf{0}\\
\mathbf{0} & I^{\prime}+C(M^{\prime})^{T}C(M^{\prime})
\end{pmatrix}%
\begin{pmatrix}
1 & \mathbf{0}\\
-\mathbf{1} & I^{\prime}%
\end{pmatrix}
\]
\[
 =
\begin{pmatrix}
1 & -\mathbf{1}\\
\mathbf{0} & I^{\prime}%
\end{pmatrix}%
\begin{pmatrix}
1 & \mathbf{0}\\
\mathbf{0} & U
\end{pmatrix}%
\begin{pmatrix}
0 & \mathbf{0}\\
\mathbf{0} & I^{\prime}+C(M)^{T}C(M)
\end{pmatrix}%
\begin{pmatrix}
1 & \mathbf{0}\\
\mathbf{0} & U^{T}%
\end{pmatrix}%
\begin{pmatrix}
1 & \mathbf{0}\\
-\mathbf{1} & I^{\prime}%
\end{pmatrix}
\]
\[
 =
\begin{pmatrix}
1 & -\mathbf{1}\\
\mathbf{0} & I^{\prime}%
\end{pmatrix}%
\begin{pmatrix}
1 & \mathbf{0}\\
\mathbf{0} & U
\end{pmatrix}%
\begin{pmatrix}
1 & \mathbf{1}\\
\mathbf{0} & I^{\prime}%
\end{pmatrix}
L^{\ast}(G)%
\begin{pmatrix}
1 & \mathbf{0}\\
\mathbf{1} & I^{\prime}%
\end{pmatrix}%
\begin{pmatrix}
1 & \mathbf{0}\\
\mathbf{0} & U^{T}%
\end{pmatrix}%
\begin{pmatrix}
1 & \mathbf{0}\\
-\mathbf{1} & I^{\prime}%
\end{pmatrix}
\text{.}%
\]

If $M$ is a maximal forest of $G$ and $e$ is a bridge of $G$ then $e\in
E(M)$ and $e$ does not appear in any circuit of $G$, so every entry of
the $e$ column of $F(M)$ is $0$. It follows that $F(M)F(M)^{T}$ is exactly the
same as the reduced dual Laplacian matrix $F(M-e)F(M-e)^{T}$ of $G-e$. This is
property II$^{\ast}$.

As mentioned above, the rows of $F(M)$ correspond to fundamental circuits with
respect to $M$. Circuit-cutset duality is reflected in the fact that the rows
of $%
\begin{pmatrix}
I & C(M)
\end{pmatrix}
$ correspond to fundamental cuts with respect to $M$, and this fact implies that
$%
\begin{pmatrix}
I & C(M)
\end{pmatrix}
\cdot
\begin{pmatrix}
I & C(M)
\end{pmatrix}^T=I+C(M)C(M)^{T}
$
is a Gram matrix for the lattice of integral cuts of $G$.
See \cite[Lemma 2]{BHN}
or \cite[Theorem 14.2.4]{GR} for a detailed discussion.


\section{Theorems \ref{Watkins}, \ref{trace} and \ref{duals}}
\label{sec:center}

The following matrix result will be useful. The argument is adapted from \cite{Wa1}.

\begin{lemma}
\label{center}Let $A$ be an $n\times m$ integer matrix, and let $a$ be the number of nonzero
columns in $A$. Then either of these two conditions implies the other.

\begin{enumerate}
\item There are a directed graph $\vec{G}$ and a unimodular matrix $C$ such
that $CA=N(\vec{G})$.

\item There are a symmetric integer matrix $B$ and a unimodular matrix $C$
such that $B=CAA^{T}C^{T}$, the row sum of $B$ is $\mathbf{0}$, and $Tr(B)\leq2a$.
\end{enumerate}

If $C$ satisfies one condition then $C$ also satisfies the other condition.
Moreover, every matrix $B$ in condition 2 has $Tr(B)=2a$.
\end{lemma}

\begin{proof}
For the implication $1\implies2$, suppose $C$ is unimodular and $N(\vec
{G})=CA$. Then the number $a$ of nonzero columns of $A$ is the same as the
number $m-\ell$ of nonzero columns of $N(\vec{G})$. The matrix $B=CAA^{T}%
C^{T}=L(G)$ has row sum $\mathbf{0}$ and trace $Tr(B)=2a=2(m-\ell)$ by properties IV
and V of unreduced Laplacian matrices.

For $2\implies1$, suppose $C$ is unimodular and $B=CAA^{T}C^{T}$ has row sum
$\mathbf{0}$ and trace $Tr(B)\leq2a$. If $\mathbf{1}$ denotes a vector whose
entries are all $1$ then $\mathbf{1}\cdot B=\mathbf{0}$, because the rows of $B$ sum to
$\mathbf{0}$. Hence $0=\mathbf{1}\cdot B\cdot\mathbf{1=1}\cdot CAA^{T}C^{T}%
\cdot\mathbf{1}=(\mathbf{1}\cdot CA)\cdot(\mathbf{1}\cdot CA)^{T}%
\mathbf{=}\left\Vert \mathbf{1}\cdot CA\right\Vert ^{2}$, so $\mathbf{1}\cdot
CA=\mathbf{0}$. That is, the rows of $CA$ sum to $\mathbf{0}$. It follows that each nonzero
column of $CA$ has at least one positive entry and at least one negative entry.

As $C$ is nonsingular, $A$ and $CA$ both have $a$ nonzero columns. The trace
$Tr(B)=Tr(CA\cdot(CA)^{T})$ is the sum of the squares of the entries of $CA$,
so since every nonzero column of $CA$ has at least two nonzero entries,
\[
Tr(B)=\sum_{i,j}(CA)_{ij}^{2}{}\geq2a\text{,}%
\]
with equality only if every nonzero column of $CA$ has exactly two nonzero
entries, both of absolute value $1$.

The hypothesis $Tr(B)\leq2a$ implies that the equality $Tr(B)=2a$ holds. The
rows of $CA$ sum to $\mathbf{0}$, so it follows that every nonzero column of $CA$ has
exactly two nonzero entries, $+1$ and $-1$. That is, $CA$ is the incidence
matrix of a directed graph.
\end{proof}

\subsection{Proof of Theorem \ref{Watkins}}

Recall that property X includes Theorem \ref{Watkins}. As discussed in Section \ref{sec:lap}, the equivalence of conditions 2, 3 and
4 of property X is well known. The implication $2\implies1$ follows
immediately from property VII. For $1\implies2$, suppose $G_{1}$ and $G_{2}$
are graphs each of which has $\ell$ loops, and suppose $U$ is a unimodular
matrix with $UL_{V_{01}}(G_{1})U^{T}=L_{V_{02}}(G_{2})$. We may assume without
loss of generality that $\left\vert E(G_{2})\right\vert =m_{2}\leq
m_{1}=\left\vert E(G_{1})\right\vert $.

Let $G_{1}^{\prime}$ be the connected graph obtained from $G_{1}$ by
identifying all the vertices of $V_{01}$ to a single vertex $v_{1}$, and let
$G_{2}^{\prime}$ be the connected graph obtained from $G_{2}$ by identifying
all the vertices of $V_{02}$ to a single vertex $w_{1}$. Let $V_{01}^{\prime
}=\{v_{1}\}$ and $V_{02}^{\prime}=\{w_{1}\}$. Then $L_{V_{01}}(G_{1}%
)=L_{V_{01}^{\prime}}(G_{1}^{\prime})$ and $L_{V_{02}}(G_{2})=L_{V_{02}%
^{\prime}}(G_{2}^{\prime})$, so $UL_{V_{01}^{\prime}}(G_{1}^{\prime}%
)U^{T}=L_{V_{02}^{\prime}}(G_{2}^{\prime})$. Let%
\[
W=%
\begin{pmatrix}
1 & \mathbf{-1}\\
\mathbf{0} & I
\end{pmatrix}
\text{, \ \ }X=%
\begin{pmatrix}
1 & \mathbf{0}\\
\mathbf{0} & U
\end{pmatrix}
\text{ \ \ and \ \ }Y=%
\begin{pmatrix}
1 & \mathbf{1}\\
\mathbf{0} & I
\end{pmatrix}
\text{,}%
\]
where $I$ is an identity matrix. Let $Z=WXY$, and order the vertices of
$V(G_{1}^{\prime})$ and $V(G_{2}^{\prime})$ with $v_{1}$ and $w_{1}$ first
(respectively). Then
\[
ZL(G_{1}^{\prime})Z^{T}=WX%
\begin{pmatrix}
0 & \mathbf{0}\\
\mathbf{0} & L_{V_{01}^{\prime}}(G_{1})
\end{pmatrix}
X^{T}W^{T}=W%
\begin{pmatrix}
0 & \mathbf{0}\\
\mathbf{0} & L_{V_{02}^{\prime}}(G_{2})
\end{pmatrix}
W^{T}\text{,}%
\]
which is $L(G_{2}^{\prime})$. Let $B=L(G_{2}^{\prime})=ZN(\vec{G}_{1}^{\prime})N(\vec{G}_{1}^{\prime})^{T}Z^{T}%
$. Properties IV and V of Laplacian matrices tell us that the row sum of $B$
is $\mathbf{0}$ and $Tr(B)=2(m_{2}-\ell)\leq2(m_{1}-\ell)$, which is twice the number
of nonzero columns of $N(\vec{G}_{1}^{\prime})$. Applying Lemma \ref{center}
with $A=N(\vec{G}_{1}^{\prime})$ and $C=Z$, we conclude that there is a
directed graph $\vec{G}_{3}$ such that $ZN(\vec{G}_{1}^{\prime})=N(\vec{G}%
_{3})$. Moreover,%
\[
L(G_{3})=N(\vec{G}_{3})N(\vec{G}_{3})^{T}=ZN(\vec{G}_{1}^{\prime})\left(
ZN(\vec{G}_{1}^{\prime})\right)  ^{T}=ZL(G_{1}^{\prime})Z^{T}=L(G_{2}^{\prime
})\text{,}%
\]
so property III tells us that $G_{3}$ is isomorphic to
$G_{2}^{\prime}$, except possibly for the placement of loops. Loop placement
does not affect incidence matrices, so we conclude that $ZN(\vec{G}%
_{1}^{\prime})=N(\vec{G}_{2}^{\prime})$, i.e., the unreduced incidence
matrices of $\vec{G}_{1}^{\prime}$ and $\vec{G}_{2}^{\prime}$ are strictly row
equivalent over $\mathbb{Z}$. It follows that $N_{V_{01}^{\prime}}(\vec{G}%
_{1}^{\prime})$ and $N_{V_{02}^{\prime}}(\vec{G}_{2}^{\prime})$ are also
strictly row equivalent over $\mathbb{Z}$; these are the same matrices as
$N_{V_{01}}(\vec G_{1})$ and $N_{V_{02}}(\vec G_{2})$.


\subsection{The unreduced version of Theorem \ref{Watkins}}

The unreduced version of Theorem \ref{Watkins} is not so different from the reduced version, but we provide details for the sake of completeness.

\begin{lemma}
\label{ranklem}
The rank of $L(G)$ over $\mathbb{Q}$ is $n-c(G)$.
\end{lemma}
\begin{proof}Properties IV and VI tell us that the rank of $L(G)$ over $\mathbb{Q}$ is no more than $n-c(G)$. The matrix-tree theorem tells us that a reduced Laplacian of $G$ is a nonsingular submatrix of $L(G)$, of order $n-c(G)$.
\end{proof} 

\begin{lemma}
\label{reducelem}
If $V_0$ includes one vertex from each connected component of $G$ then $L(G)$ is congruent over $\mathbb{Z}$ to the matrix
\[
\begin{pmatrix}
L_{V_0}(G) & \mathbf{0}\\
\mathbf{0} & \mathbf{0}
\end{pmatrix}
\text{.}
\]
\end{lemma}
\begin{proof}Properties IV and VI tell us that the displayed matrix is $UL(G)U^T$, where $U$ is obtained from an identity matrix by changing the $vw$ entry to $1$ whenever $v \in V_0$, $v \neq w$ and $v,w$ lie in the same connected component of $G$.
\end{proof}

\begin{corollary}
\label{xcor}
Suppose the unreduced Laplacian matrices of $G_{1}$ and $G_{2}$ are congruent over $\mathbb{Z}$. Then the reduced Laplacian matrices of $G_{1}$ and $G_{2}$ are congruent over $\mathbb{Z}$.
\end{corollary}

\begin{proof} As $L(G_{1})$ and $L(G_{2})$ are congruent over $\mathbb{Z}$, Lemma~\ref{reducelem} tells us that
\[
A_1=
\begin{pmatrix}
L_{V_{01}}(G_1) & \mathbf{0}\\
\mathbf{0} & \mathbf{0}
\end{pmatrix}
\text{\ \  and \ \ }
A_2=
\begin{pmatrix}
L_{V_{02}}(G_2) & \mathbf{0}\\
\mathbf{0} & \mathbf{0}
\end{pmatrix}
\]
are also congruent over $\mathbb{Z}$. Hence there is a unimodular matrix $U$ with $UA_1 U^T=A_2$. Also, the fact that $L(G_1)$ and $L(G_2)$ are congruent implies that they have the same rank; so according to Lemma~\ref{ranklem}, $L_{V_{01}}(G_1)$ and $L_{V_{02}}(G_2)$ have the same size. It follows that
\[
UA_1U^T=
\begin{pmatrix}
U_1 & U_2\\
U_3 & U_4
\end{pmatrix}
\begin{pmatrix}
L_{V_{01}}(G_1) & \mathbf{0}\\
\mathbf{0} & \mathbf{0}
\end{pmatrix}
\begin{pmatrix}
U_{1}^T & U_{3}^T\\
U_{2}^T & U_{4}^T
\end{pmatrix}
=A_2
\]
requires $U_1 L_{V_{01}}(G_1) U_{1}^T=L_{V_{02}}(G_2)$ and $U_1 L_{V_{01}}(G_1) U_{3}^T=\mathbf{0}$.

As $L_{V_{01}}(G_1)$ and $L_{V_{02}}(G_2)$ are both nonsingular, $U_1 L_{V_{01}}(G_1) U_{1}^T=L_{V_{02}}(G_2)$ implies that $U_1$ is nonsingular too. Then $U_1 L_{V_{01}}(G_1) U_{3}^T=\mathbf{0}$ implies that $U_{3}=\mathbf{0}$, so $\det(U)=\det(U_1)\det(U_4)$. Necessarily then $U_1$ is unimodular, so $U_1 L_{V_{01}}(G_1) U_{1}^T=L_{V_{02}}(G_2)$ implies that $L_{V_{01}}(G_1)$ and $L_{V_{02}}(G_2)$ are congruent over $\mathbb{Z}$.
\end{proof}

Here is the unreduced version of Theorem \ref{Watkins}.
\begin{proposition}Let $G_{1}$ and $G_{2}$ be graphs with the same number of loops. Then $L(G_{1})$ and $L(G_{2})$ are congruent over $\mathbb{Z}$ if and only if $G_{1}$ and $G_{2}$ are 2-isomorphic graphs with the same number of vertices and the same number of connected components.
\end{proposition}

\begin{proof}If $L(G_{1})$ and $L(G_{2})$ are congruent over $\mathbb{Z}$, then they certainly have the same rank and size. It follows that $G_{1}$ and $G_{2}$ have the same values for $n-c(G)$ and $n$, so $G_{1}$ and $G_{2}$ have the same number of vertices and the same number of connected components. Corollary~\ref{xcor} and Theorem \ref{Watkins} tell us that $G_{1}$ and $G_{2}$ are 2-isomorphic.

For the converse, suppose $G_{1}$ and $G_{2}$ are 2-isomorphic graphs with the same number of vertices and the same number of connected components. Then the matrices
\[
A_1=
\begin{pmatrix}
L_{V_{01}}(G_1) & \mathbf{0}\\
\mathbf{0} & \mathbf{0}
\end{pmatrix}
\text{\ \  and \ \ }
A_2=
\begin{pmatrix}
L_{V_{02}}(G_2) & \mathbf{0}\\
\mathbf{0} & \mathbf{0}
\end{pmatrix}
\]
have the same size. Theorem \ref{Watkins} tells us that $L_{V_{01}}(G_1)$ and $L_{V_{02}}(G_2)$ are congruent over $\mathbb{Z}$, so $A_1$ and $A_2$ are congruent over $\mathbb{Z}$. According to Lemma~\ref{reducelem}, it follows that $L(G_1)$ and $L(G_2)$ are congruent over $\mathbb{Z}$.
\end{proof}

\subsection{Proof of Theorem \ref{trace}}
\label{subsec}

Let $G$ be a graph with $m$ edges and $b$ bridges. Recall that Theorem \ref{trace} has two parts. 1. If $L^{\ast}(G)$ is an unreduced dual Laplacian
matrix of $G$ then $Tr(L^{\ast}(G))$ is an even integer $\geq2(m-b)$. 2. $G$
is planar if and only if $G$ has an unreduced dual Laplacian matrix with
$Tr(L^{\ast}(G))=2(m-b)$.

It is easy to verify that $Tr(L^{\ast}(G))$ is an even integer. The row sum of $L^{\ast}(G)$ is $\mathbf{0}$, so the sum of the entries of $L^{\ast}(G)$ is $0$. It follows that $-Tr(L^{\ast}(G))$ is the sum of the off-diagonal entries of $L^{\ast}(G)$; this sum is even because $L^{\ast}(G)$ is symmetric.

It is also easy to verify one direction of part 2. If $G$ is planar then $G$ has an
abstract dual $G^{\ast}$, and Theorem \ref{duals} tells us that $L(G^{\ast})$
is an unreduced dual Laplacian matrix of $G$. (Theorem \ref{duals} is proven
below; there is no circularity because the proof does not involve Theorem \ref{trace}.) As $G^{\ast}$ has $m$ edges and $b$ loops, property V guarantees
that $Tr(L(G^{\ast}))=2(m-b)$.

We verify part 1 and the other direction of part 2 simultaneously, by proving
that if $L^{\ast}(G)$ is an unreduced dual Laplacian matrix of $G$ with
$Tr(L^{\ast}(G))\leq2(m-b)$ then $Tr(L^{\ast}(G))=2(m-b)$ and $G$ is planar.

Suppose that $G$ is a graph with an unreduced dual Laplacian matrix $L^{\ast
}(G)$ such that $Tr(L^{\ast}(G))\leq2(m-b)$. According to Definitions
\ref{lapdual} and \ref{unredlapdual}, $G$ has a maximal forest $M$ such that
$L^{\ast}(G)$ is obtained from a matrix congruent to $I^{\prime}+C(M)^{T}C(M)$
by adjoining a row and column to make the row and column sums equal to $\mathbf{0}$.
Let
\[
F(M)=%
\begin{pmatrix}
C(M)^{T} & -I^{\prime}%
\end{pmatrix}
\cdot P_{M}%
\]
as in Section \ref{sec:dlap}, and let $A$ be the matrix obtained from $F(M)$
by adjoining a new first row with all entries equal to $0$. Then the number of
nonzero columns of $A$ is the same as the number of nonzero columns of $F(M)$,
and according to the discussion in\ Section \ref{sec:dlapgen} this is the
number of edges of $G$ that appear in circuits of $G$. That is, $A$ has
$a=m-b$ nonzero columns.

Suppose $U$ is unimodular and $L^{\ast}(G)$ is obtained by adjoining a row and column to $U(I^{\prime
}+C(M)^{T}C(M))U^{T}$, in such a way that the row and column
sums equal $\mathbf{0}$. Then we have%
\begin{gather*}
L^{\ast}(G)=%
\begin{pmatrix}
1 & -\mathbf{1}\\
\mathbf{0} & I^{\prime}%
\end{pmatrix}%
\begin{pmatrix}
0 & \mathbf{0}\\
\mathbf{0} & U(I^{\prime}+C(M)^{T}C(M))U^{T}%
\end{pmatrix}%
\begin{pmatrix}
1 & \mathbf{0}\\
-\mathbf{1} & I^{\prime}%
\end{pmatrix}
\\
=ZAA^{T}Z^{T}\text{, }\text{where \ }Z=%
\begin{pmatrix}
1 & -\mathbf{1}\\
\mathbf{0} & I^{\prime}%
\end{pmatrix}%
\begin{pmatrix}
1 & \mathbf{0}\\
\mathbf{0} & U
\end{pmatrix}
=%
\begin{pmatrix}
1 & -\mathbf{1}\cdot U\\
\mathbf{0} & U
\end{pmatrix}
\text{.}%
\end{gather*}
Then $A$, $B=L^{\ast}(G)$ and $C=Z$ satisfy part 2 of Lemma \ref{center}, so
the lemma guarantees that $Tr(L^{\ast}(G))=2a=2(m-b)$ and there is a directed
graph $\vec{G}^{\ast}$ such that $ZA=N(\vec{G}^{\ast})$. The group generated
by the rows of $N(\vec{G}^{\ast})$ is the group of cuts of $G^{\ast}$, and as
noted at the beginning of Section \ref{sec:dlapgen}, the group generated by
the rows of $F(M)$ is the group of cycles of $G$. The equation $ZA=N(\vec
{G}^{\ast})$ implies that these two groups are the same, so if $\beta
:E(G)\rightarrow E(G^{\ast})$ is the bijection that matches edges according to
the correspondence between columns of $A$ and $N(\vec{G}^{\ast})$, then cuts
of $G^{\ast}$ correspond to cycles of $G$ under $\beta$. That is, $G$ and
$G^{\ast}$ are abstract duals; hence both are planar.


\subsection{Theorem \ref{trace} and MacLane's criterion}

The planarity criterion of MacLane \cite{Ma} is this: $G$ is planar if and
only if there is a $GF(2)$ basis for its cycle space, in which each edge
appears no more than twice. If we augment such a basis with one more element,
equal (modulo 2) to the sum of the basis elements, then the resulting set has
the property that every non-bridge edge appears precisely twice.

In one direction, the relationship with Theorem \ref{trace} is simple.  If $L^{\ast}(G)$ is an unreduced dual Laplacian matrix of $G$, then there is a $\mathbb{Z}$ basis $B$ of the group of cycles
of $G$, such that $L^{\ast}(G)$ records the dot products among the vectors in
the set $B^{\prime}$ obtained by augmenting $B$ with one more element, equal
to the negative of the sum of the elements of $B$. Notice that every
non-bridge edge of $G$ is represented at least once among the elements of
$B$, and at least twice among the elements of $B^{\prime}$. Each diagonal
entry of $L^{\ast}(G)$ is a positive integer, at least as large as the number
of edges represented in the corresponding element of $B^{\prime}$. (A diagonal
entry will be larger than the number of edges represented in the corresponding
element of $B^{\prime}$ if the absolute value of some coordinate of that
element is more than $1$.) It follows that $Tr(L^{\ast}(G))\geq2(m-b)$, with
equality only if each non-bridge edge is represented in precisely two
elements of $B^{\prime}$. Clearly then $Tr(L^{\ast}(G))=2(m-b)$ implies that
$B$ satisfies MacLane's criterion.

The opposite direction is not so immediate, as MacLane's criterion provides
only a $GF(2)$ basis, not a $\mathbb{Z}$ basis. Of course if $G$ satisfies
MacLane's criterion then $G$ is planar, and it is easy to verify Theorem \ref{trace} for planar graphs, as indicated in Subsection \ref{subsec}.

\subsection{Proof of Theorem \ref{duals}}

Let $G$ be a planar graph. Theorem \ref{duals} asserts that these two
statements about a graph $G^{\ast}$ are equivalent. 1. $G$ and $G^{\ast}$ are
abstract duals. 2. The number of loops in $G^{\ast}$ is the same as the number of bridges
in $G$, and a reduced Laplacian matrix of $G^{\ast}$ is a reduced dual Laplacian matrix of $G$.

If $G$ and $G^{\ast}$ are abstract duals then there are oriented versions
$\vec{G},\vec{G}^{\ast}$ and a bijection $\beta:E(G)\rightarrow E(G^{\ast})$
under which the cycle vectors of $G$ correspond to the cut vectors of
$G^{\ast}$. If we match the columns of $F(M)$ and $N_{V_{0}^{\ast}}(\vec
{G}^{\ast})$ according to $\beta$, then the rows of $F(M)$ and $N_{V_{0}%
^{\ast}}(\vec{G}^{\ast})$ generate the same group. Recall that $F(M)$ has $m-n+c(G)$ rows by definition, and the number of rows in $N_{V_{0}^{\ast}}(\vec{G}^{\ast})$ is the same as the number of edges in a maximal forest of $G^{\ast}$. As $G$ and $G^{\ast}$ are abstract duals, the number of edges in a maximal forest of $G^{\ast}$ is $m-\left\vert E(M) \right\vert=m-(n-c(G))$, the same as the number of rows in $F(M)$. It follows that $F(M)$ and
$N_{V_{0}^{\ast}}(\vec{G}^{\ast})$ are strictly row equivalent over
$\mathbb{Z}$, so there is a unimodular $U$ with $N_{V_{0}^{\ast}}(\vec
{G}^{\ast})=UF(M)$. Then $L_{V_{0}^{\ast}}(G^{\ast})=N_{V_{0}^{\ast}}(\vec
{G}^{\ast})N_{V_{0}^{\ast}}(\vec{G}^{\ast})^{T}=UF(M)F(M)^{T}U^{T}$ is a
reduced dual Laplacian matrix of $G$. Also, the number of loops in $G^{\ast}$
is the number of $\mathbf{0}$ columns of $N_{V_{0}^{\ast}}(\vec{G}^{\ast})$, and the
number of bridges in $G$ is the number of $\mathbf{0}$ columns of $F(M)$; if the
matrices are row equivalent these numbers must be equal. This verifies the
implication $1\implies2$.

Suppose condition 2 holds. As $G$ is planar, it has an abstract dual $D$.
Applying the implication $1\implies2$ to $D$ in place of $G^{\ast}$, we
conclude that the number of loops in $G^{\ast}$ is the same as the number of
loops in $D$, and both $L_{V_{0}^{\ast}}(G^{\ast})$ and a reduced Laplacian of
$D$ are reduced dual Laplacians of $G$. But then
$L_{V_{0}^{\ast}}(G^{\ast})$ and a reduced Laplacian of $D$ are congruent to
each other over $\mathbb{Z}$, so Theorem \ref{Watkins} tells us that $G^{\ast}$ and $D$
are 2-isomorphic. As $D$ is an abstract dual of $G$, so is $G^{\ast}$. 
 
\section{Two examples}
\label{sec:examples}

\begin{figure}[bht]
\centering

\begin{tikzpicture} [scale=1]
\filldraw
(0,0) circle (2pt)
(0,2) circle (2pt)
(1.7,1) circle (2pt)
(3.7,1) circle (2pt)
(5.4,2) circle (2pt)
(5.4,0) circle (2pt)
(7.1,1) circle (2pt);
\node at (0,-0.3) {$v_1$};
\node at (0,2.3) {$v_2$};
\node at (1.8,0.7) {$v_3$};
\node at (3.6,0.7) {$v_4$};
\node at (5.4,-0.3) {$v_5$};
\node at (5.4,2.3) {$v_6$};
\node at (7.43,1) {$v_7$};

\begin{scope}[decoration={
    markings,
    mark=at position 0.53 with {\arrow{>}}}
    ]
    \draw [postaction={decorate}, very thick] (0,0)--(0,2);
    \draw [postaction={decorate}, very thick] (0,0)--(1.7,1);
    \draw [postaction={decorate}, semithick] (0,2)--(1.7,1);
    \draw [postaction={decorate}, very thick] (1.7,1)--(3.7,1);
    \draw [postaction={decorate}, very thick] (3.7,1)--(5.4,0);
    \draw [postaction={decorate}, very thick] (3.7,1)--(5.4,2);
    \draw [postaction={decorate}, semithick] (5.4,0)--(5.4,2);
    \draw [postaction={decorate}, very thick] (5.4,0)--(7.1,1);
    \draw [postaction={decorate}, semithick] (5.4,2)--(7.1,1);
\end{scope}
\node at (-0.3,1) {$e_1$};
\node at (0.95,0.2) {$e_2$};
\node at (0.95,1.75) {$e_3$};
\node at (2.7,1.3) {$e_4$};
\node at (4.5,0.2) {$e_5$};
\node at (4.5,1.8) {$e_6$};
\node at (5.1,1) {$e_7$};
\node at (6.35,0.2) {$e_8$};
\node at (6.35,1.75) {$e_9$};

\end{tikzpicture}

\caption{The graph $G$ in Example 1.}%
\label{laplanf1}%
\end{figure}
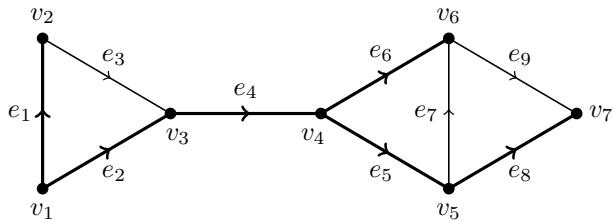

\textbf{Example 1} Suppose $G$ is the graph pictured in Figure \ref{laplanf1},
with bold edges indicating the spanning tree $M$ with $E(M)=\{e_{1}%
,e_{2},e_{4},e_{5},e_{6},e_{8}\}$. Using the indicated edge directions and
$V_{0}=\{v_{3}\}$, we obtain%
\[
\bordermatrix{
& e_1 & e_2 & e_3 & e_4 & e_5 & e_6 & e_7 & e_8 & e_9 \cr
v_1 & -1 & -1 & 0 & 0 & 0 & 0 & 0 & 0 & 0 \cr
v_2 & 1 & 0 & -1 & 0 & 0 & 0 & 0 & 0 & 0 \cr
v_3 & 0 & 1 & 1 & -1 & 0 & 0 & 0 & 0 & 0 \cr
v_4 & 0 & 0 & 0 & 1 & -1 & -1 & 0 & 0 & 0 \cr
v_5 & 0 & 0 & 0 & 0 & 1 & 0 & -1 & -1 & 0 \cr
v_6 & 0 & 0 & 0 & 0 & 0 & 1 & 1 & 0 & -1 \cr
v_7 & 0 & 0 & 0 & 0 & 0 & 0 & 0 & 1 & 1 \cr
}=N(\vec{G})\text{,}%
\]%
\[
\text{so \ \ }%
\begin{pmatrix}
-1 & 1 & -1 & 0 & 0 & 0 & 0 & 0 & 0\\
0 & 0 & 0 & 0 & -1 & 1 & -1 & 0 & 0\\
0 & 0 & 0 & 0 & 1 & -1 & 0 & 1 & -1
\end{pmatrix}
=F(M)
\]
\[
\text{ \ and \ \ }%
\begin{pmatrix}
6 & -3 & -1 & -2\\
-3 & 3 & 0 & 0\\
-1 & 0 & 3 & -2\\
-2 & 0 & -2 & 4
\end{pmatrix}
=\widehat{F}(M)\widehat{F}(M)^{T}\text{,}%
\]
in the notation of Section \ref{sec:dlap}.

\begin{figure}[bht]%
\centering

\begin{tikzpicture} [scale=1.5]

\filldraw
(0,0) circle (1.33pt)
(0,2) circle (1.33pt)
(3,0) circle (1.33pt)
(3,2) circle (1.33pt);
\node at (-0.2,-0.13) {$v_1$};
\node at (-0.3,2) {$v_2$};
\node at (3.2,2.13) {$v_3$};
\node at (3.22,-0.15) {$v_4$};

\begin{scope}[decoration={
    markings,
    mark=at position 0.51 with {\arrow{>}}}
    ]
    \draw [postaction={decorate}, semithick] (0,2) .. controls (-0.6,1.3) and (-0.6,0.7) .. (0,0);
    \draw [postaction={decorate}, semithick] (0,0) -- (0,2);
    \draw [postaction={decorate}, very thick] (0,2) .. controls (0.4,1.3) and (0.4,0.7) .. (0,0);
    \draw [postaction={decorate}, semithick] (0,2) .. controls (-0.6,2.6) and (0.6,2.6) .. (0,2);
    \draw [postaction={decorate}, semithick] (3,2) .. controls (2.7,1.3) and (2.7,0.7) .. (3,0);
    \draw [postaction={decorate}, semithick] (3,0) .. controls (3.3,0.7) and (3.3,1.3) .. (3,2);
    \draw [postaction={decorate}, very thick] (3,2) -- (0,0);
    \draw [postaction={decorate}, semithick] (0,0) .. controls (1,-0.3) and (2,-0.3) .. (3,0);
    \draw [postaction={decorate}, very thick] (3,0) .. controls (2,0.3) and (1,0.3) .. (0,0);
\end{scope}
\node at (-0.62,1) {$e_1$};
\node at (-0.18,1) {$e_2$};
\node at (0.5,1) {$e_3$};
\node at (0.35,2.3) {$e_4$};
\node at (2.58,1) {$e_5$};
\node at (3.42,1) {$e_6$};
\node at (1.5,1.22) {$e_7$};
\node at (1.5,-0.4) {$e_8$};
\node at (1.5,0.42) {$e_9$};

\end{tikzpicture}

\caption{The graph $G^{\ast}$ in Example 1.}%
\label{laplanf2}%
\end{figure}

Notice that the trace of $\widehat{F}(M)\widehat{F}(M)^{T}$ is $16=2(m-b)$, so
in the notation of Subsection \ref{subsec}, $U$ can be taken to be an
identity matrix. As predicted by the argument of Subsection \ref{subsec}, it
turns out that $ZA=\widehat{F}(M)$ is the incidence matrix of a graph $G^{\ast}$.
This graph is pictured in Figure \ref{laplanf2}, with bold edges indicating
the spanning tree $M^{\ast}$ with $E(M^{\ast})=\{e_{3},e_{7},e_{9}\}$. It is
not difficult to verify that $G^{\ast}$ is an abstract dual of $G$, but it
happens that the two graphs are not geometric duals, i.e., they cannot be
drawn together in the plane in such a way that each graph has one vertex in
each complementary region of the other graph. One way to see this is to
observe that there is no vertex of $G^{\ast}$ incident only on $e_{7}$,
$e_{8}$ and $e_{9}$, but every drawing of $G$ has a complementary region with
boundary $\{e_{7},e_{8},e_{9}\}$. This example illustrates the fact that Theorem \ref{duals} involves abstract rather than geometric duality.%

\textbf{Example 2} In Example 1 it happens that $G$ has a maximal forest $M$
such that $Tr(\widehat{F}(M)\widehat{F}(M)^{T})=2(m-b)$. That is, the
planarity criterion of Theorem \ref{trace} is satisfied by an unreduced dual
Laplacian matrix obtained directly from a matrix of the form $I^{\prime
}+C(M)^{T}C(M)$. It is not always the case that Theorem \ref{trace} is
satisfied so readily. For instance, consider the graph $G$ of Figure
\ref{laplanf3}. Then $G$ has $m=9$ edges, none of which is a bridge. As $G$
has $5$ vertices, an unreduced dual Laplacian matrix $L^{\ast}(G)$ is a
$6\times6$ matrix. The diagonal entries of $L^{\ast}(G)$ are the dot products
with themselves of certain nonzero elements of the cycle group of $G$, and the
smallest cycles of $G$ are of length $3$, so if $Tr(L^{\ast}(G))=18$ then each
diagonal entry of $L^{\ast}(G)$ must correspond to a 3-cycle of $G$. This is
not possible for an $\widehat{F}(M)\widehat{F}(M)^{T}$ matrix, because the
$-I^{\prime}$ block of $F(M)$ guarantees that the row adjoined to $F(M)$ in
constructing $\widehat{F}(M)$ has more than $m-n+c(G)=5$ nonzero entries. We
leave it as an exercise for the reader to verify that nevertheless, $G$ does
have an unreduced dual Laplacian matrix $L^{\ast}(G)$ with $Tr(L^{\ast
}(G))=18$.

\begin{figure}[h]%
\centering

\begin{tikzpicture} [scale=1]
\filldraw
(0,0.5) circle (2pt)
(0,2) circle (2pt)
(2,2) circle (2pt)
(2,0.5) circle (2pt)
(1,3) circle (2pt);

\draw [semithick] (0,0.5) -- (0,2) -- (2,2) -- (2,0.5) -- (0,0.5);
\draw [semithick] (0,0.5) -- (2,2) -- (1,3) -- (0,2);
\draw [semithick] (0,0.5) .. controls (-1.5,2) and (-0.5,3) .. (1,3);
\draw [semithick] (2,0.5) .. controls (3.5,2) and (2.5,3) .. (1,3);

\end{tikzpicture}

\caption{Example 2.}%
\label{laplanf3}%
\end{figure}
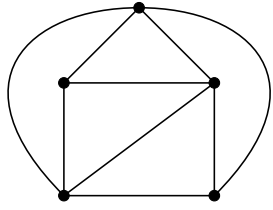

\end{document}